\documentclass[
final
, nomarks
]{dmtcs-episciences}

\usepackage[utf8]{inputenc}

\newtheorem{theorem}{Theorem}[section]
\newtheorem{proposition}[theorem]{Proposition}
\newtheorem{lemma}[theorem]{Lemma}

\newtheorem{corollary}[theorem]{Corollary}

\usepackage[numbers]{natbib}

\author{Mikl\'os B\'ona \affiliationmark{1}
  }
\title[Vertex rank in phylogenetic trees]{On the number of vertices of each rank in \\
 $k$-phylogenetic trees}

\affiliation{
  Department of Mathematics,
  University of Florida, Gainesville, FL, USA}
\received{2015-06-10}
\revised{2015-11-4}
\accepted{2016-03-17}
\begin{document}
\publicationdetails{18}{2016}{3}{7}{653}
\maketitle

\begin{abstract}
  We find surprisingly simple formulas for the limiting probability that the rank of a randomly selected vertex
in a randomly selected  $k$-phylogenetic tree  is a given integer. 
\end{abstract}

\section{Introduction}
Various parameters of many models of random rooted trees are fairly well understood {\em if they relate to a near-root part of the tree or to global tree structure}. The first group includes, for instance, the numbers of vertices at given distances from the root, the immediate progeny sizes
for vertices near the top, and so on. See \cite{flajoletsedgewick} for a comprehensive treatment of these results. The tree height and width are parameters of global nature, see  \cite{kol,dev,mahmoudpit,averagerootheight,kespit,growtrees,cheon,mansour} for instance. In recent years there has
been a growing interest in analysis of the random tree fringe, i. e. the  tree part close to the leaves, 
 \cite{aldous,mahmoudward1,mahmoudward2,du,protected,treeav,janson-holmgren,janson-holmgren-2,janson}. These articles  either focused on unlabeled trees, or trees in which every vertex was labeled. 

In this paper, we study another natural class of trees, those in which {\em only the leaves are labeled}. Some trees of this kind have been studied from different aspects. See \cite{flajolet2009is} for a result of the present author and Philip Flajolet on the subject, or Chapter 5 of \cite{stanley} for enumerative results for two tree varieties of this class.  

First, we will consider {\em $k$-phylogenetic trees}, that is, rooted non-plane trees whose vertices are bijectively labeled 
with the elements of the set $[n]=\{1,2,\cdots ,n\}$, and in which each non-leaf vertex has exactly $k$ children. See Figure \ref{3phylotrees} for the set of all three 2-phylogenetic trees on label set $[3]$. 

\begin{figure} \label{3phylotrees}
\begin{center}
  \includegraphics[width=60mm]{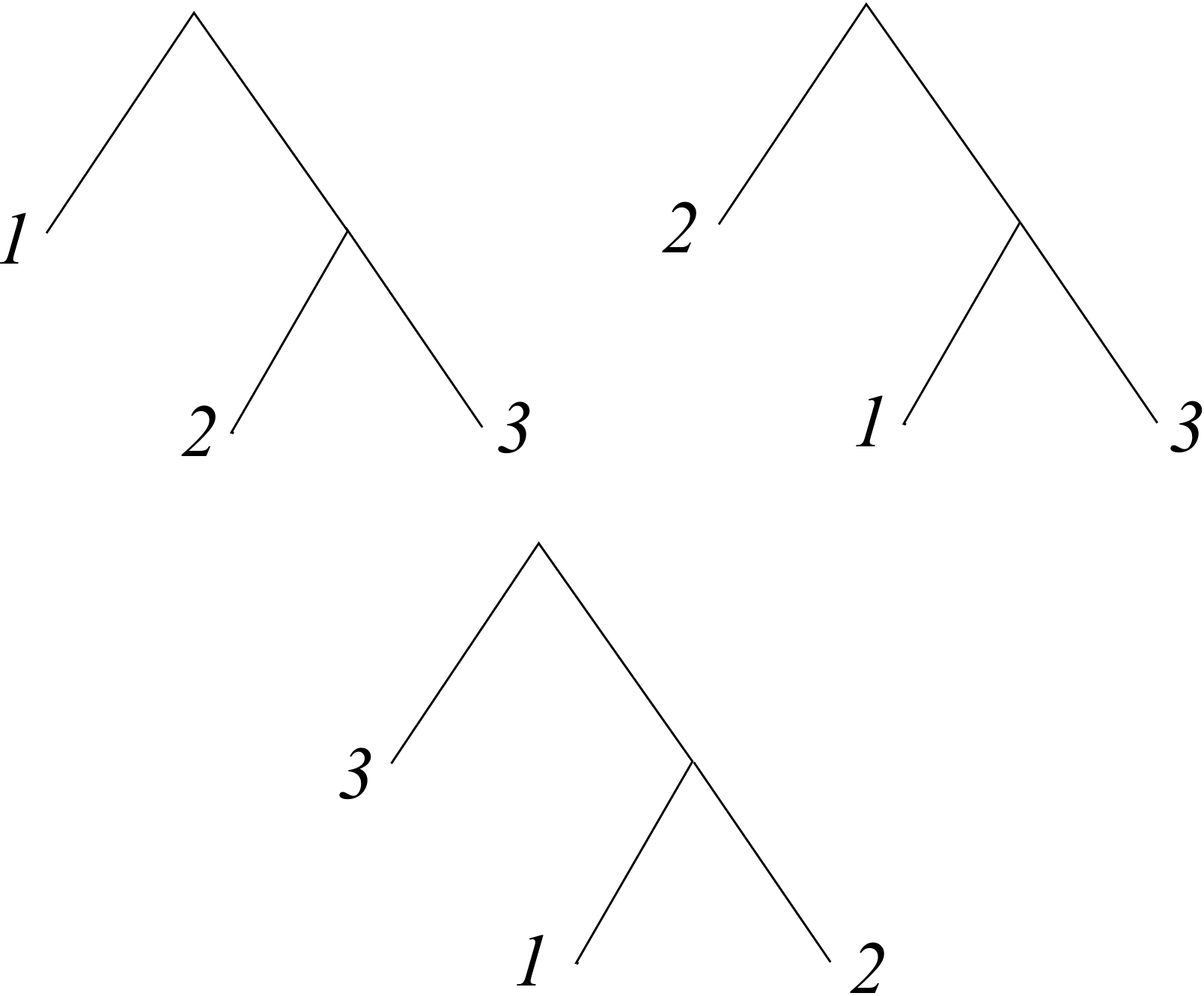}
  \caption{The three 2-phylogenetic trees on leaf set $[3]$.}
  \end{center}
 \end{figure}

We define the {\em rank} of a vertex as the distance of that vertex from its closest descendent leaf, so leaves have rank 0, neighbors of leaves have rank 1, and so on. Then for each fixed $i$, we are able to prove that as $n$ goes to infinity, the probability that a random vertex of a random phylogenetic tree on label set $[n]$ is of rank $i$ converges to a limit $P_{k,i}$, and we are able to compute that limit. The obtained numerical values will be much simpler than the numerical values obtained for other tree varieties, for instance in \cite{protected} or \cite{treeav}. Indeed, we will prove that \[P_{k,i}=k^{-c_i} -k^{-1-kc_i},\] where $c_i=c_{i,k}=(k^i-1)/(k-1)$.  This will follow from an even simpler formula  for the probability that a random vertex in a random $k$-phylogenetic tree is of rank {\em at least $i$}. 
  
The Lagrange inversion formula will be our main tool. 

These results are notable for several reasons. First, the obtained formulas are surprisingly simple. Second, the numbers $P_{k,i}$ decrease very fast, in a doubly exponential way. To compare, note that in \cite{treeav}, the corresponding numbers for binary search trees are shown to decrease in a simply exponential way. Third, the obtained explicit formulas make it routine to prove that the sequence $P_{k,i}$ is log-concave for any fixed $i$, a fact that is plausible to conjecture, but probably hopeless to prove, for many other tree varieties. Fourth, in the last section we will show an example to illustrate that even for 2-phylogenetic trees, there are similar questions that lead to much more complicated numerical
answers, so the simplicity of our results is surprising. 

We end the paper by a few open questions, asking for combinatorial proofs of some of the mentioned phenomena. 

\section{Enumeration}
\subsection{Our trees and the Lagrange inversion formula} \label{lagrange}
 Let $t_{k,n}$ be the number of $k$-phylogenetic trees
on leaf set $[n]$, and set $t_{k,0}=0$. Let $T_k(x)=\sum_{n\geq 0}t_{k,n}\frac{x^n}{n!}$ be the exponential generating function of the sequence of  these numbers.

Removing the root of such a tree, we get either the empty set, or an unordered set of $k$ such trees, leading
to the functional equation
\begin{equation} \label{funceq} T_k(x) = x + \frac{T_k^k(x)}{k!}. \end{equation} 
This means that $T_k(x)$ is the compositional inverse of the power series $F_k(x)=x-x^k/k!$, so the coefficients of 
$T_k(x)$ can be computed by the Lagrange inversion formula. However, that does not imply that the power series
$T_k(x)$ has a simple closed form. In fact, it usually does not, since it is a solution of a functional equation of degree $k$,
where $k$ can be arbitrarily high.

Let $m_{i,k}(n)$ denote the number of all vertices that are of rank at least $i$ in all $k$-phylogenetic trees on leaf set 
$[n]$.  Let  $M_{i,k}(x)$ be the exponential generating function of the numbers $m_{i,k}(n)$. Similarly let $r_{i,k}(n)$ be the number of $k$-phylogenetic trees on leaf set $[n]$ in which the root is of rank at least $i$,  and let $R_{i,k}(x)$ be the exponential generating function of the numbers $r_{i,k}(n)$.

While the Lagrange inversion formula cannot provide a closed form for most of our generating functions, it is still useful for us in that it enables us to prove  the following useful proposition.  We include the proof of the proposition, but it can be skipped without causing difficulties in reading the rest of the paper. 

\begin{proposition} \label{polynomial-small}
Let $p$ be a polynomial function. Then \[\lim_{n\rightarrow \infty} \frac{[x^n]p\left(T_k(x)\right)}{[x^n]M_{0,k}(x)} 
=0.\]
\end{proposition}

\begin{proof}
Note that  $[x^n]M_{0,k}(x)$ as well as  $[x^n]T_k(x)$, and hence, $[x^n] p\left(T_k(x)\right)$ are nonzero if and only if
$n-1$ is divisible by $k-1$. Indeed, growing a $k$-phylogenetic tree from a single root by turning leafs into parents of leaves,
each step will increase the number of leaves by $k-1$.  

Clearly, it suffices to prove the statement in the special case when $p(x)=x^{\ell}$, that is, when $p(T_k(x))=T_k^{\ell}(x)$. Indeed, 
all polynomials are linear combinations of such monomials with constant coefficients.  We can also assume that $\ell >0$, since the stament is obviously true for the polynomial $x^0=1$. 

We use the following version of the Langrange inversion formula (see Chapter 5 of \cite{stanley} for a proof).
Let $n$ and $\ell$ be positive integers, and let $F^{\langle -1 \rangle}(x)$ be the compositional inverse of the power series
$F(x)$. Then 
\begin{equation} \label{lif}   n[x^n] (F^{\langle -1 \rangle}(x))^{\ell} =\ell [x^{n-\ell}] \left (\frac{x}{F(x)}\right )^n.                                                                                                                         \end{equation}

Setting $F(x)=F_k(x)=x-\frac{x^k}{k!}$, and recalling that $F^{\langle -1 \rangle}(x)=T_k(x) $, formula (\ref{lif}) yields
\[  n[x^n] T_k^{\ell}(x) = \ell [x^{n-\ell}] \left (\frac{x}{x-\frac{x^k}{k!}}\right )^n.\] From this, we compute
\begin{eqnarray*} [x^n] T_k^{\ell}(x) & =& \frac{\ell}{n}  [x^{n-\ell}] \left(1-\frac{x^{k-1}}{k!} \right)^{-n} \\
& = & \frac{\ell}{n}  [x^{n-\ell}] \sum_{s\geq 0} {-n\choose s} \left(-\frac{x^{k-1}}{k!} \right)^s \\
& = & \frac{\ell}{n}  [x^{n-\ell}] \sum_{s\geq 0} {n+s-1\choose s} \frac{x^{s(k-1})}{k!^s}. 
\end{eqnarray*}

So, setting $n-\ell=s(k-1)$, we have $n=s(k-1)+\ell$, and the last displayed chain of equalities implies that
\begin{equation} \label{firstcoeff}  [x^n] T_k^{\ell}(x) =  \frac{\ell}{s(k-1)+\ell} {ks+\ell-1\choose s} \frac{1}{k!^s}.\end{equation} 

Note that in particular, for $\ell=1$, we get
\begin{equation} \label{firstcoeffspec}  [x^n] T_k(x) =  \frac{1}{s(k-1)+1} {ks\choose s} \frac{1}{k!^s}.\end{equation} 

On the other hand, as  $M_{0,k}(x)$ counts all vertices of  all $k$-phylogenetic trees on leaf set $[n]$. As we said at the begining of this proof, this implies that $n=(k-1)s+1$, for some nonnegative integer $s$, and it is easy to see that such trees 
have exactly $s$ non-leaf vertices, and therefore, $ks+1$ total vertices. So each coefficient of $M_{0,k}$ is $ks+1$ times
as large as the corresponding coefficient of $T_k(x)$. 

Therefore, it follows from (\ref{firstcoeffspec}) that 
\[ [x^n]M_{0,k}(x)= (ks+1)  \frac{1}{s(k-1)+1} {ks\choose s} \frac{1}{k!^s}.\]
Comparing this with (\ref{firstcoeff}), we get that
\[ \frac{[x^n] 
\left(T_k(x)^{\ell} \right)}{[x^n]M_{0,k}(x)}  = 
\frac{\frac{\ell}{s(k-1)+\ell} \cdot  {ks+\ell-1\choose s} \cdot \frac{1}{k!^s}}{(ks+1)  \cdot \frac{1}{s(k-1)+1} {ks\choose s} \cdot \frac{1}{k!^s}} \]
\[=\frac{1}{ks+1} \cdot \frac{(s(k-1)+1)\ell}{s(k-1)+\ell} \cdot \frac{(ks+\ell-1)(ks+\ell-2)\cdots (ks+\ell-s)}{(ks)(ks-1)\cdots 
(ks-s+1)} .\] As $n$ goes to infinity, so does $n-1=(k-1)s$, and therefore, $ks$. So the product in the last displayed line clearly converges to 0, since the first term converges to 0, the second one converges to the fixed integer $\ell$, and the third one converges to 1. 
\end{proof}

\subsection{Formulas for generating functions}
We will now use the tools discussed in Section \ref{lagrange} to prove some enumerative lemmas.

\begin{lemma} 
For all  integers $k\geq 2$, and for all integers $i\geq 0$, the equality
\[M_{i,k}(x)=M_{i,k}(x)\cdot \frac{T_k(x)^{k-1}}{(k-1)!}+ R_{i,k}(x)\] holds. 
\end{lemma}

\begin{proof} Removing the root of a $k$-phylogenetic tree in which one non-root vertex of rank at least $i$ is marked, we get 
one such tree with one marked vertex of rank at least $i$, and an unordered set of $k-1$ trees with no marked vertices. By the product formula of exponential generating functions, such collections have generating function $M_{i,k}(x) \cdot  \frac{T_k(x)^{k-1}}{(k-1)!}$. 
On the other hand, trees in which the root is marked and is of rank at least $i$ are simply counted by $R_{i,k}(x)$.
\end{proof}

Therefore, \begin{equation}
M_{i,k}(x)= \frac{ R_{i,k}(x)}{1- \frac{T_k(x)^{k-1}}{(k-1)!} } .\end{equation}

\begin{proposition} \label{generalrec}  For all $i\geq 1$, the recurrence relation 
\begin{equation} R_{i,k}(x)= \frac{R_{i-1,k}^k(x)}{k!} \end{equation}
holds.
\end{proposition}

\begin{proof} Removing the root of a $k$-phylogenetic tree in which the root has rank at least $i$, we get an unordered set of $k$ such trees in which the root has rank at least $i-1$. The claim now follows from the product formula. 
\end{proof}

Let us introduce the notation \[c_i=c_{i,k}=\frac{k^i-1}{k-1}\]
for shortness. 

\begin{corollary} 
For all $i\geq 0$, the equality 
\begin{equation} R_{i,k}(x)=  \frac{T_k(x)^{k^i}} {k!^ {c_i }}  \end{equation} 
holds. 
\end{corollary}

\begin{proof} This is straightforward by induction. Indeed, for $i=0$, the equality $R_{i,k}(x)=T_k(x)$ holds, since in each tree, the root is of rank at least 0. Let us assume that the statement is true for $i-1$, that is, 
\[ R_{i-1,k}(x)=  \frac{T_k(x)^{k^{i-1}} }{k!^ {c_{i-1} }} .\]
Now take the $k$th power of both sides, then divide by $k!$. By Proposition \ref{generalrec}, this turns the left-hand side into 
$R_{i,k}(x)$, so we get the equality
\[R_{i,k}(x)=\frac{T_k(x)^{k^i} }{k!^ {kc_{i-1}+1 }} .\]
This proves our claim since $kc_{i-1}+1=c_i$.
\end{proof}

\begin{corollary} For all $i\geq 0$, the equality \begin{equation}
 \label{explicitformgen} M_{i,k}(x)= \frac{1}{k!^ {c_i} }  \cdot \frac{T_k(x)^{k^i}}{1- \frac{T_k(x)^{k-1}}{(k-1)!} }  \end{equation} 
holds.

In particular, the generating function for the total number of vertices is  \begin{equation}
\label{explicitform0}
M_{0,k}(x)=  \frac{T_k(x)}{1- \frac{T_k(x)^{k-1}}{(k-1)!} } .\end{equation} 
\end{corollary}

\subsection{Our main results}
Now we are in a position to state and prove the main result of this paper. 

\begin{theorem} For all integers $k\geq 2$, and for all integers $i\geq 1$, the equality
\begin{equation} \lim_{n\rightarrow \infty} \frac{m_{i,k(n)}}{m_{0,k}(n)} = \frac{1}{k^{c_i}}=
\frac{1}{k^{\frac{k^i-1}{k-1}}} \end{equation} holds.
\end{theorem}

That is, for large $n$, about  $\frac{1}{k^{c_i}}$ of all vertices are of rank at least $i$.

\begin{proof}
We proceed by splitting a constant multiple of $M_{i,k}(x)$ into two parts,
one of which will turn out to be a constant multiple of $M_{0,k}(x)$, and the other one of which will turn out to be negligible,
 by a divisibility argument. 

To that end, we consider the rightmost factor in  (\ref{explicitformgen}), and essentially divide the numerator by the denominator,  noting that 
\begin{eqnarray*} \frac{T_k(x)^{k^i}}{1- \frac{T_k(x)^{k-1}}{(k-1)!} } & = & 
\frac{\left(\frac{T_k(x)^{(k-1)c_i} }{(k-1)!^{c_i}} -1 \right ) (k-1)!^{c_i}T_k(x) + (k-1)!^{c_i}T_k(x)}{1- \frac{T_k(x)^{k-1}}{(k-1)!} } \\
 & = &  \frac{\left( \frac{T_k(x)^{(k-1)c_i} }{(k-1)!^{c_i}} -1 \right)  (k-1)!^{c_i} T_k(x)}{1- \frac{T_k(x)^{k-1}}{(k-1)!} } + 
\frac{(k-1)!^{c_i}T_k(x)}{1- \frac{T_k(x)^{k-1}}{(k-1)!} } \\
& =&  \frac{\left( \frac{T_k(x)^{(k-1)c_i} }{(k-1)!^{c_i}} -1 \right)  (k-1)!^{c_i} T_k(x)}{1- \frac{T_k(x)^{k-1}}{(k-1)!} } + (k-1)!^{c_i}M_{0,k}(x) .\end{eqnarray*}
We have used (\ref{explicitform0}) in the last step. 

Now note that $f^{c_i}-1=(f-1)(f^{c_i-1}+f^{c_i-2}+\cdots +f +1)$. Using this formula for $f=T_k(x)^{k-1} /(k-1)!$, we see that the first summand of the last line in the last displayed array of equations is a {\em polynomial} function of $T_k(x)$, 
that is, we have proved that 
\[ \frac{T_k(x)^{k^i}}{1- \frac{T_k(x)^{k-1}}{(k-1)!} } =p\left(T_k(x)\right) +(k-1)!^{c_i} M_{0,k}(x) .\]
By Proposition \ref{polynomial-small}, the contribution of $p\left(T_k(x)\right) $ to the coefficient of $x^n$ on the right-hand side is negligible. Comparing this observation with (\ref{explicitformgen}) completes the proof. 
\end{proof}

\begin{corollary} Then for each fixed $i$, as $n$ goes to infinity, the probability that a random vertex of a random $k$-phylogenetic tree on label set $[n]$ is of rank $i$ converges to a limit $P_{k,i}$, and
\[P_{k,i} =  \frac{1}{k^{c_i}} -  \frac{1}{k^{c_{i+1}}}=\frac{1}{k^{c_i}}-\frac{1}{k^{kc_i+1}}.\]
\end{corollary}

\bibliographystyle{abbrv}
\bibliography{finalphylotrees}
\label{sec:biblio}

\end{document}